\documentclass[12pt]{amsart}
\title[Carath\'eodory convergence and type problem]{
Carath\'eodory convergence and the conformal type problem}
\author[A. Eremenko]{Alexandre Eremenko}
\address{Mathematics Department, Purdue University,
West Lafayette, IN 47907, USA}
\email{eremenko@purdue.edu}
\author[S.~Merenkov]{Sergei Merenkov}
\address{Department of Mathematics, The City College of New York and CUNY Graduate Center, New York, NY 10031, USA}
\email{smerenkov@ccny.cuny.edu}
\thanks{S.~M.\ supported by NSF grant DMS-2247364.}

\usepackage{amsmath, amssymb, amsthm,latexsym}
\usepackage{graphicx}
\newcommand\C{{\mathbb C}}
\newcommand\oC{\overline {\mathbb C}}

\newcommand\N{{\mathbb N}}
\newcommand\D{{\mathbb D}}

\newcommand\R{{\mathbb R}}

\renewcommand\:{\colon}

\newtheorem{theorem}{Theorem}[section]
\newtheorem{definition}[theorem]{Definition}

\newtheorem{proposition}[theorem]{Proposition}

\newtheorem{remark}{Remark}
\newtheorem{lemma}[theorem]{Lemma}

\theoremstyle{definition}

\begin{document}

\abstract{We study Carath\'eodory convergence for open, simply connected surfaces spread over the sphere and, in particular, provide examples demonstrating that in the Speiser class
	the conformal type can change when two singular values collide.
}
\endabstract

\maketitle

\section{Introduction}\label{S:Intro}
\noindent
Carath\'eodory Kernel Convergence is an important tool in the theory of univalent functions. It gives a geometric criterion for a sequence of normalized univalent functions in the unit disk to converge uniformly on compacta, and gives a description of the image of the disk under the limiting map. 
In this paper we adapt the notion of convergence in the sense of Carath\'eodory, introduced in C.~Carath\'eodory~\cite{Ca12}, see also L.~I.~Volkovyski\u\i~\cite{Vo48}, to the setting of  pointed surfaces spread over the sphere. Moreover, we establish results, Theorems~\ref{T:CC},~\ref{T:CC2}, that relate such convergence to convergence on compacta omitting certain exceptional sets. A result similar to the necessary part of Theorem~\ref{T:CC} for surfaces spread over the plane was proved by K.~Biswas and R.~Perez-Marco~\cite[Theorem~1.2]{BPM15}.  Another aim of this paper is to provide examples of sequences of open, simply connected surfaces spread over the sphere (in fact, over the plane) that have only finitely many singular values and  whose conformal type changes when two of the singular values collide; see Section~\ref{S:TC}. We also give an example of a sequence of entire functions in the plane with finitely many values, so that each function in the sequence has infinite order, while the limit has order one; see Section~\ref{S:CO}.   

\subsection{Surfaces spread over the sphere}\label{SS:SSS}
Classically, Riemann surfaces are thought of as surfaces associated to holomorphic or, more generally, meromorphic functions.
\begin{definition}
A surface spread over the sphere is a pair $(S,f)$, where $S$ is an open topological surface and $f\: S\to\oC$ is a continuous, open and discrete map, called a projection. Here, $\oC$ is the Riemann sphere.
\end{definition}
The surface $S$ can be endowed with the pull-back conformal structure, so that $f$ becomes holomorphic. In what follows, we do not distinguish two surfaces $(S_1,f_1)$ and $(S_2,f_2)$ if there exists a homeomorphism $h\: S_1\to S_2$ such that 
$$
f_1=f_2\circ h.
$$ 
The homeomorphism $h$ is conformal when $S_1$ and $S_2$ are endowed with the pull-back conformal structures.
If $S$ is an open,  simply connected surface, then, equipped with the pull-back conformal structure, it is equivalent to either the complex plane $\C$  or the unit disk $\D$ in $\C$. In the former case we call $(S,f)$ \emph{parabolic}, and in the latter \emph{hyperbolic}. For a survey on surfaces spread over the sphere and the type problem one can consult~\cite{Er21}. 

Near each point, a continuous, open and discrete map $f$ is homeomorphically equivalent to the map $z\mapsto z^d$, where $d\in\N$. More precisely, for each $p_0\in S$, there exists an open neighborhood $U$ of $p_0$ and two homeomorphisms $h_1, h_2$, such that $h_1\: U\to\D,\ h_2\:f(U)\to\D$, with $h_1(p_0)=0,\ h_2(f(p_0))=0$, and 
$$h_2\circ f\circ h_1^{-1}(z)=z^d,\quad z\in\D.
$$ 
The number $d$ is called the \emph{local degree} of $f$ at $p_0$. It does not depend on the choice of homeomorphisms $h_1$ and $h_2$. If $d>1$, $p_0$ is called the \emph{critical point} and $f(p_0)$ the \emph{critical value} of $f$. 
An element $a\in\oC$ is called an \emph{asymptotic value} of $f$ if there exists a path $\gamma\:[0,1)\to S$, called an \emph{asymptotic tract}, that leaves every compact set of $S$ as $t\to1$, and such that $$\lim_{t\to1}f(\gamma(t))=a.$$ A \emph{singular} value of $f$ is either a critical or an asymptotic value.

\begin{definition}
A surface spread over the sphere $(S,f)$ is said to belong to Speiser class $\mathcal S$, if the projection map $f$ has only finitely  many singular values.
\end{definition}

Examples of surfaces from Speiser class include $(\C,p)$, where $p$ is an arbitrary polynomial, $(\C,\exp)$, $(\C,\sin)$, $(\C,\cos)$, $(\C,\exp\circ\exp)$, $(\C,\wp)$, where $\wp$ is the Weierstrass $\wp$-function, $(\D,\lambda)$, where $\lambda$ is the modular function, etc. 
Surfaces from Speiser class have combinatorial descriptions in terms of labeled Speiser graphs as follows. Let $\beta$ be a \emph{base curve}, i.e., a curve in the sphere that contains all singular values of $f$. For example, when all singular values of $f$ are real or $\infty$, we can chose $\beta$ to be the extended real line. Then, $\beta$ divides the sphere $\oC$ into two topological hemispheres, which we can call an ``upper" and ``lower" hemispheres for convenience. We fix two points, one in each of the hemispheres, and denote the point in the upper hemisphere by $\times$ and the one in the lower hemisphere by $\circ$. If the number of singular values of $f$ is $k$, we look at the graph $G_\beta$ in $\oC$ with two vertices $\times$ and $\circ$ and $k$ edges, each edge connecting $\times$ to $\circ$ in such a way that it separates two adjacent singular values on $\beta$. The \emph{Speiser graph} $G_{f, \beta}$ of $(S,f)$ is then the preimage of $G_\beta$ under the map $f$. It is bipartite, homogeneous of degree $k$, and its faces, i.e., connected components of the complement in $S$, are labeled by the corresponding singular values of $f$. The labels appear in cyclic orders around each vertex of $G_{f,\beta}$, according to the order of the singular values on $\beta$ viewed from $\times$ or $\circ$, respectively. See Figures~\ref{fig:EE}, \ref{fig:EEE}, \ref{fig:E} for examples. Conversely, given such a labeled graph $G$ in the plane and a base curve $\beta$ that contains all the labels of the faces, one can reconstruct a surface $(S, f)$ by gluing together upper and lower hemispheres of $\oC\setminus\{\beta\}$ identifying them along the arcs of $\beta$ between adjacent labels according to the combinatorics of the graph $G$.      
 See~\cite{GO70} or~\cite{Ne70} for further details.    


\subsection{Carath\'eodory convergence}
We consider 
triples $T=(S,f,w)$, where $(S,f)$ is a surface spread over the sphere, and $w\in S$ a point which is not critical for $f$. 
We refer to these triples as \emph{pointed} surfaces spread over the sphere. Two pointed surfaces spread over the sphere $T_1=(S_1,f_1, w_1)$ and $T_2=(S_2, f_2, w_2)$ are \emph{equivalent}, denoted $T_1\sim T_2$, if there exists a homeomorphism $h\: S_1\to S_2$ such that $h(w_1)=w_2$ and $f_1=f_2\circ h$. Equivalence classes
are still called \emph{pointed surfaces spread over the sphere}.
In addition to the equivalence,
we define the order relation on surfaces spread over the sphere:
$$(S_1,f_1,w_1)\subset (S_2,f_2,w_2)$$
means that there is a continuous injective map $\phi\: S_1\to S_2$
such that 
\begin{equation}\label{1}
\phi(w_1)=w_2\quad\mbox{and}\quad f_1=f_2\circ\phi.
\end{equation}
The second equation in (\ref{1}) implies that $\phi$ is holomorphic.
It is easy to see
that $T_1\subset T_2$ and 
$T_2\subset T_1$ imply that there is a homeomorphism $h\: S_1\to S_2$
satisfying (\ref{1}) with $\phi=h$. In this case $T_1\sim T_2$.  

Each equivalence class contains a {\em normalized} triple
with
$S=D_R:=\{ z\in\C:|z|<R\}$ for some $R\in(0,+\infty]$, 
$w=0$, and $f^{\#}(0)=1$, where $f^{\#}$ is the 
spherical derivative,
$$f^{\#}=\frac{f'}{1+|f|^2}.$$
A triple $T$ is called maximal if $T\subset T_1$
implies that $T\sim T_1$. 
If $S=D_R$, the open disk of radius $R$ centered at the origin, the maximality means that $f$ has no meromorphic continuation
beyond $D_R$.
\vspace{.1in}

C.~Carath\'eodory~\cite{Ca12} and L.~I.~Volkovyski\u\i~\cite{Vo48} defined convergence of Riemann surfaces generalizing Carath\'eodory convergence for sequences of univalent functions. 
The following two definitions are adapted from~\cite{Vo48}. As in~\cite{Vo48}, in these definitions and below, we assume that if $(S_n,f_n,w_n),\ n\in\N$, is a sequence of pointed surfaces spread over the sphere, then there are $p\in\oC$ and  $r>0$ such that for every $n\in\N$ there exists a domain $W_n$ in $S_n$ containing $w_n$, such that $f_n(w_n)=p$ and $f_n\: W_n\to B(p,r)$ is a homeomorphism.
\noindent
\begin{definition}\label{D:Ker} 
A kernel of a sequence $(S_n,f_n,w_n),\ n\in\N$, of pointed surfaces spread over the sphere is a pointed surface $(S, f, w)$, such that:

\vspace{.2cm}
\noindent
1) There exists a discrete set $E\in S$ with $w\not\in E$,
and for every compact
$K\in S\backslash E$ such that $w\in K$ there exists  $N\in\N$
with the property that for each $n>N$ there exists
a continuous embedding $\phi_{K,n}:K\to S_n$
with $\phi_{K,n}(w)=w_n$ and  $f_n\circ\phi_{K,n}=f$.
The set $E$ is called an exceptional set.

\vspace{.2cm}
\noindent
2) The triple $(S,f, w)$ satisfying property 1) is maximal in the sense of the order relation defined above.  
\end{definition}

The necessity of having an exceptional set $E$ is demonstrated by the following examples. Let $S_n$ be the surfaces obtained by gluing two spheres with slits $[1,1+1/n]$, identifying the lower edge of the slit on one of the spheres with the upper edge of the other using the identity map. Each $S_n,\ n\in\N$, is a simply connected Riemann surface of genus 0 that can be endowed with an obvious projection $f_n$ onto the sphere. We further select one of the two points projecting to 0 as a marked point $w_n$, making $(S_n, f_n, w_n)$ into a sequence of pointed surfaces spread over the sphere. The kernel of such a sequence $(S_n, f_n, w_n)$ as $n\to\infty$ is the sphere with the identity map, and the exceptional set $E$  is $\{1\}$. We can also modify the above sequence $S_n$ to make each surface to be a topological plane by removing one of the points projecting to $\infty$. Another, analytic, example is the following. Let $S_n=\C$ and $f_n(z)=z^2+1/n$. Then $(\C, f_n, w_n),\ n\in\N$, where $w_n=\sqrt{1-1/n}$, has Carath\'eodory kernel $(\C,f, 1)$, with $f(z)=z^2$, and the exceptional set $E$ is $\{0\}$.

\begin{definition}
A sequence $(S_n, f_n, w_n),\ n\in\N$, converges to $(S, f, w)$ in the sense of Carath\'eodory if $(S, f, w)$ is the kernel of every subsequence of $(S_n, f_n, w_n)$.
\end{definition}

It is clear that this definition is compatible with the equivalence
relation on pointed surfaces spread over the sphere.
This notion of convergence is also a version of the one defined in~\cite{BPM15}, adapted to surfaces spread over the sphere rather than the plane. 
Indeed, the relation $f_n\circ\phi_{K,n}=f$ implies that the continuous embedding $\phi_{K,n}$ is, in fact, an isometric embedding when surfaces are endowed with the pull-back spherical metric $d$. Specifically, $d$ is a path metric on $S$ given in some local coordinate $z=\sigma(f(p))\in\C$ of every non-critical point $p_0\in S$ by $2|dz|/\left(1+|z|^2\right)$, where $\sigma$ is the stereographic projection. The metric space $(S,d)$ is not complete in the presence of asymptotic values of $f$.

\section{Properties of the Carath\'eodory kernel}\label{S:PCK}
\noindent
Simple examples contained in ~\cite{Vo48} show that not every sequence of pointed surfaces has a subsequence converging in the sense of Ca\-ra\-th\'e\-o\-do\-ry to the kernel of the whole sequence. One such example is obtained by choosing a countable dense collection $\{a_1, a_2,\dots\}$ of points on the unit circle and letting $S_n,\ n\in\N$, to be the plane with the radial slit from $a_n$ to $\infty$. We choose 0 as the marked point for each $S_n$ and the projection map $f_n$ is the identity for each $n$. The kernel of the whole sequence is the open unit disk. However, from any sequence of such surfaces one can select a subsequence $S_{n_k},\ k\in\N$, with the corresponding $a_{n_k},\ k\in\N$, converging to $a$. The kernel of such a subsequence is the plane with the closed radial ray emanating from $a$ removed, which is different from the unit disk. Therefore, such a sequence of pointed Riemann surfaces does not converge in the sense of Carath\'eodory. 

In~\cite{Tr52}, Yu.~Yu.~Trohim\v cuk gave elementary examples when kernels of sequences of surfaces may not be unique. One example is obtained as follows. Let each $S_{2k-1},\ k\in\N$, be the disk $\{|z-1|<2\}$, and each $S_{2k},\ k\in\N$, the two-sheeted disk over  $\{|z-1|<2\}$ with single branch point at $z=1$. For odd $n$, the map $f_n$ on $S_n$ is the identity, and for even $n$, it is the projection. We choose $w_n$ to be 0 if $n$ is odd and one of the 2 points projecting to 0 for even $n$. Then, for any Jordan arc $J$ connecting $z=1$ to its boundary and avoiding $\{|z|<1\}$, the surface  $\{|z-1|<2\}\setminus J$ with the identity map and $w=0$ is a kernel of the whole sequence $S_n,\ n\in\N$. Such a sequence does not converge in the sense of Carath\'eodory either. 

In the same paper, Trohim\v cuk gave the following characterization for  uniqueness of a kernel. Assume that $(S_n, f_n, w_n),\ n\in\N$, is a sequence of pointed surfaces spread over the sphere as above, i.e.,  
 there exist $p\in\oC$ and $r>0$ such that for each $n\in\N$ there exists a domain $W_n$ in $S_n$ containing $w_n$, such that $f_n\: W_n\to B(p,r)$ is a homeomorphism, $f_n(w_n)=p$.
We say that a parametrized  curve $\gamma$ in $\oC$ with the initial point  $p$ is \emph{admissible} for $(S_n, f_n, w_n)$ if there exists a chain of disks $B(p_i, r_i),\ i=1,2,\dots,k$, in $\oC$ covering $\gamma$ with $p_1=p$ and $p_{i+1}\in\gamma\cap B(p_i, r_i),\ i=1,2,\dots, k-1$, corresponding to the increasing sequence of the parameter,  such that the following holds:
for each $i=1,2,\dots, k$, there exist $N_i\in\N$, and for all $n>N_i$ a domain  $W_{n,i}\subset S_n$,
$\ W_{n,1}=W_n$, 
with $f_n\: W_{n,i}\to B(p_i,r_i),\ n>N_i$, being a homeomorphism, and there exist  $p_{n,i}\in W_{n,i-1}\cap W_{n,i}$ with $f_n(p_{n,i})=p_i,\ i=2, 3,\dots, k$. Since there are only finitely many disks in the above definition, there exists $N_\gamma\in\N$ such that $W_{\gamma,n}=\cup_{i=1}^k W_{n,i}\subset S_n$ for all $n> N_\gamma$ and $W_{\gamma,n}$ is a domain, i.e., an open connected set. Each map $f_n\: W_{\gamma,n}\to\cup_{i=1}^k B(p_i, r_i),\ n>N_\gamma$, is a covering, but not necessarily a homeomorphism. 
An admissible parametrized  curve $\gamma$ is called \emph{normal} if there exists a finite collection of disks $B(p_i, r_i),\ i=1,2,\dots, k$, covering  $\gamma$ as above and $N\in\N$, such that for all $m,n>N$ there exists a homeomorphism $\phi_{\gamma,m,n}\: W_{\gamma, m}\to W_{\gamma,n}$ with $f_n\circ \phi_{\gamma,m,n}= f_m$ on $W_{\gamma,m}$.  

\begin{theorem}\cite[Theorem~2]{Tr52}\label{T:Troh}
For $(S_n,f_n,w_n),\ n\in\N$, to have a unique kernel it is necessary and sufficient that every admissible  curve is normal.
\end{theorem}

We use this theorem to show uniqueness of kernel in the Speiser class.

\begin{proposition}\label{P:UK}
If $(S_n, f_n, w_n),\ n\in\N$, is a sequence of surfaces having a surface of Speiser class $\mathcal S$ as its kernel, then the kernel 
is unique, i.e., two kernels are equivalent in the above sense. 
\end{proposition}
\begin{proof}
Let $(S,f,w)$ be a kernel for $(S_n, f_n, w_n),\ n\in\N$, where $(S,f)\in\mathcal S$.
Let $\gamma\in\oC$ be an arbitrary admissible curve,
and let $\tilde\gamma$ be a
lift under $f$ of a maximal subcurve of $\gamma$ such that $\tilde\gamma$ has initial point $w$ and does not contain critical points of $f$ other than possibly at the other end point. We claim that $f(\tilde\gamma)=\gamma$, i.e., that the full lift of $\gamma$ does not contain any critical points of $f$. Indeed, if the terminal point of $\tilde\gamma$ is a critical point of $f$, look at a small circle $C$ (in the pull-back of the spherical metric) centered at it. Let $K$ be the compact set that is the union of the subcurve of $\tilde\gamma$ from $w$ to the first intersection of $\tilde\gamma$ with $C$ and $C$. If $K$ contains points of an exceptional set $E$ in the definition of kernel, we perturb it slightly such that the resulting compact set, still denoted $K$, does not intersect $E$. Thus, there exists $N\in\N$ such that for all $n\ge N$ we have a continuous embedding $\phi_{K,n}$ from $K$ into $S_n$ such that $\phi_{K,n}(w)=w_n$ and  $$f_n\circ\phi_{K,n}=f$$ on $K$. This is a contradiction since there is $N_\gamma\in\N$ such that $f_n,\ n> N_\gamma$, do not have critical points along the lift of $\gamma$ starting at $w_n$ because $\gamma$ is admissible, and so the left-hand side of the last equation is one-to-one on $C$ but the right-hand side is not. 

Now we cover $\gamma$ by disks $B(p_i,r_i),\ i=1,2,\dots, k$,  as in the definition of admissible curves. By passing to smaller disks if necessary, we can  assume that the closure of $W_\gamma$ defined for $f$ in the same way as $W_{n,\gamma}$ for $f_n$ above, does not contain critical points. If $W_\gamma$ contains other elements of $E$, we remove small open disks of radii $\delta$ around such points and denote the resulting set by $W_{\gamma,\delta}$.  The closure  $\overline{W_{\gamma,\delta}}$ being compact in $S\setminus E$ implies that there exists $N\in\N$ such that for all $n>N$ there is a continuous embedding $\phi_{\gamma,\delta, n}$ of  $\overline{W_{\gamma,\delta}}$ into $S_n$ with $f_n\circ \phi_{\gamma,\delta, n}=f$. These equations show that $\phi_{\gamma,\delta, n}$ has a holomorphic extension to an embedding of punctured neighborhoods of points in $E\cap \overline{W_\gamma}$, and removability gives an extension to an embedding of all of  $\overline{W_\gamma}$.
Therefore, if $m,n>N$, then we can choose $\phi_{\gamma,m,n}=\phi_{\gamma, n}\circ \phi_{\gamma, m}^{-1}$, and so $\gamma$ is normal and the proof is complete.    
\end{proof}

\begin{remark}
A consequence of Proposition~\ref{P:UK} is that if a sequence of surfaces $(S_n, f_n, w_n),\ n\in\N$, converges in the sense of Carath\'eodory to a surface in class $\mathcal S$, then its limit is unique. A similar statement and its proof are contained in~\cite[Proposition~3.2]{BPM15}. 
\end{remark}

\section{Uniform Convergence on Compacta}
\noindent
Notice that if all triples $(D_R,f,0),(D_{R_n},f_n,0),\ n\in\N$, are normalized, $(D_R,f,0)$ is a kernel of $(D_{R_n},f_n,0),\ n\in\N$, 
and if $K$ is a compact in $D_R$ containing $0$ in its interior, then $\phi_{K,n}'(0)=1$ for all $n>N$, where $\phi_{K,n}$ is from Definition~\ref{D:Ker}.

If $(S,f,w)$ is a simply connected surface spread over the sphere, its \emph{conformal radius} is the unique $R\le+\infty$ such that there exists a conformal map $F\: S\to D_R(0)$ with $F(w)=0,\ F'(w)=1$. The normalization makes $R$ well defined, and we call $F$ the \emph{normalized uniformizing map} of $(S,f,w)$. The following theorem is an extension of a result from~\cite{BPM15} to the spherical metric case. 

\noindent
\begin{theorem}\label{T:CC} 
Let $R\ge+\infty$ and $(D_{R_n},f_n,0),\ n\in\N$, be a sequence of normalized pointed open simply connected surfaces spread over the sphere satisfying $\limsup R_n\le R$. The triple $(D_{R_n},f_n,0)$ converges to a normalized triple $(D_R,f,0)$
in the sense of Carath\'eodory if and only if there exists an exceptional set $E$ in $D_R$ such that for
$(D_{R_n},f_n,0)$ one has:
\vspace{.1in}

\noindent
a) $f_n(0)=f(0)$ for all $n\in\N$, 
\vspace{.1in}

\noindent
b)  $\lim_{n\to\infty}R_n=R$,  and
\vspace{.1in}

\noindent
c) $f_n$ converge to $f$ uniformly on every compact set in $D_R\backslash E$.
\end{theorem}

To prove this theorem, we need the following lemma.
\begin{lemma}\label{L:UB} 
{Let $D$ be a region in the plane, containing $0$.
For a positive constant $C$, 
let $\mathcal F_C$ be the family of all univalent holomorphic functions in $D$ satisfying
$f(0)=0,\ |f'(0)|\leq C$. Then $\mathcal F_C$ is uniformly bounded on each compact
subset of $D$.}
\end{lemma}
\begin{proof} The Koebe Distortion Theorem guarantees the conclusion for the case that $D$ is a disk centered at $0$. Moreover, in this case the derivatives of $f\in \mathcal F_C$ are also uniformly bounded on compacta. 
To prove it for general $D$, we consider an arbitrary
point $z_0\in D$. Then there is a finite sequence of disks
$B_k,\; 0\leq k\leq n$, all contained in $D$ and such that
$B_0$ is centered at $0$, $B_n$ contains $z_0$,
and for every $k=1,2,\dots,n$, the center
of $B_k$ belongs to $B_{k-1}$. Applying the above result for disks
successively to $B_0,B_1,\ldots,B_n$, we obtain the statement of the lemma.
\end{proof}

\vspace{.1in}
\noindent
{\em Proof of  Theorem~\ref{T:CC}.}
We start with the sufficiency.
Conditions a), b), and c) are satisfied for any subsequence of $(D_{R_n}, f_n, 0),\ n\in\N$, so, to simplify notations, we may assume that the whole sequence is such a subsequence.
 We choose $E'$ to be the union of $E$ and all the critical points of $f$. This is a discrete subset of $(D_R,f,0)$. Let $K$ be an arbitrary compact in $D_R\setminus E'$. By making it bigger, we can always assume that it has the form $K=\{z\: |z|\le R_0\}\setminus E'_\delta$, where $E'_\delta$ is an open $\delta$-neighborhood of $E'$. 
Let $N$ be chosen so large that each $D_{R_n},\ n>N$, contains $K$, each $f_n,\ n>N$, has no critical points in $K$, and, if $z\in E'$ is a critical point of $f$ of multiplicity $m$, then each $f_n,\ n>N$, has total multiplicity of critical points in the $\delta$-neighborhood of $z$ equal $m$. The last condition is guaranteed by an application of Rouch\'e's Theorem. We can now choose $\phi_{K,n}=f_n^{-1}\circ f$, where the inverse branch of $f_n$ is chosen so that $\phi_{K,n}(0)=0$. The above conditions on $N$ guarantee that each $\phi_{K,n},\ n>N$, is one-to-one analytic in a neighborhood of $K$. The maximality of $(D_R, f,0)$ is trivial.

The proof of necessity is similar to that of~\cite[Theorems~1.1, 1.2]{BPM15}. 
Part a) follows from the definition of Carath\'eodory convergence. To prove parts b) and c),
it is enough to show that each subsequence of $(D_{R_n}, f_n, 0)$ has a further subsequence that satisfies b), and c). To simplify notations, we again assume that $(D_{R_n}, f_n, 0)$ is already a subsequence.
Let $E$ be the exceptional set from the definition of Carath\'eodory convergence, which we may assume contains all the critical points of $f$. 
For each compact $K$ in $D_R\setminus E$ and $n$ large enough, let $\phi_{K, n}\: K\to D_{R_n}$ be a continuous embedding such that $f_n\circ\phi_{K,n}=f$. Note that the normalization of $f$ and $f_n$ imply that $\phi_{K,n}(0)=0,\ \phi_{K,n}'(0)=1$. 
Exhausting $D_R\setminus E$ by compacta $K_j,\ j\in\N$, we obtain a sequence  $\phi_{K_j,n_j},\ j\in\N$, of univalent holomorphic maps in the interiors of the respective compacta, whose domains contain a fixed neighborhood of 0 and exhaust $D_R\setminus E$. Also, they are normalized by $\phi_{K_j,n_j}(0)=0,\ \phi_{K_j,n_j}'(0)=1$. Lemma~\ref{L:UB}  implies that, given any compact set $K$ in $D_R\setminus E$, a subsequence of $\phi_{K_j,n_j},\ j\in\N$, is uniformly bounded on $K$. Therefore, using a diagonalization argument we obtain that a subsequence of $\phi_{K_j,n_j},\ j\in\N$, converges uniformly on compacta to a conformal map $\phi$ in $D_R\setminus E$. The assumption that $\limsup R_n\le R$ implies that the image of $D_R\setminus E$ under $\phi$ is contained in $\overline{D_R}$.  Since $E$ is discrete, it is removable for $\phi$ and we continue to denote the continuous extension of $\phi$ to $E$ by $\phi$. Note that $\phi$ satisfies $\phi(0)=0,\ \phi'(0)=1$. If $R=\infty$, Liouville's Theorem implies that $\phi(\C)=\C$ and the normalization gives that $\phi$ is the identity. If $R<+\infty$, the Schwarz Lemma gives that $\phi$ is the identity. 
\qed

\vspace{.1in}

The following example demonstrates the difficulty of establishing uniform convergence without the assumption $\limsup R_n\le R$, e.g.,  in the case of parabolic surfaces converging to a hyperbolic one. 

Let $D_1$ and $D_2$ be two distinct simply connected domains containing 0,
and let the Riemann maps $g_j$ of $D_j,\ j=1,2$, onto the unit disk
be normalized by $g_j(0)=0$, and $g_j'(0)=1,\ j=1,2$. In addition, we assume that $g_1, g_2$ have analytic extensions to the plane as entire functions.
Let $f$ be an analytic function in the unit disk which has no analytic continuation to a bigger domain, $f(0)=0,\ f'(0)=1$,
and let $f_n$ be its Taylor partial sums. 
For example, we can take $f$ to be the normalized conformal map of the unit disk onto a domain bounded by von Koch snowflake. 
Consider now the sequence of entire functions  $h_n$ given by $h_{2k-1}=f_k\circ g_1$ and $h_{2k}=f_k\circ g_2,\ k=1, 2,\dots$.
The Carath\'eodory kernel for this sequence will be the Riemann surface of $f$, which is hyperbolic.
The sequence is normalized, consists of entire functions in the whole plane, but no limit of functions $h_n,\ n\in\N$, exists since $D_1\neq D_2$.
Based on this example, we do not expect that a general uniform  convergence statement can be proved for a sequence of parabolic surfaces converging to a hyperbolic one in the sense of Carath\'eodory.
However, if we pass to subsurfaces, we have the following result.

\begin{theorem}\label{T:CC2} 
Let $(S_n, f_n, w_n),\ n\in\N$, be as in Figure~\ref{fig:EE}, where we choose
$a=a_n\to\infty$. Then there exists a
sequence $(S_n', f_n, w_n),\ n\in\N$, where $S_n'\subset S_n$ is open and simply connected, that has the following properties: 
 $(S_n', f_n, w_n),\ n\in\N$, has the same kernel as $(S_n, f_n, w_n),\ n\in\N$, which is $(S, f, w)$  as in Figure~\ref{fig:EEE},
and
for the normalized uniformizing maps $F_n:
S_n'\to D_{R_n}$, we have $\lim R_n=R<+\infty$, where $R$ is the conformal radius of $(S, f, w)$, and $F_n,\ n\in\N$, converge uniformly on compacta  to $F\: S\to D_R$, the normalized uniformizing map of $(S, f, w)$.
\end{theorem}
\begin{proof}
By Theorem~\ref{T:PtoHtoP} below, the sequence $(S_n, f_n, w_n),\ n\in\N$, converges to $(S, f, w)$ in the sense of Carath\'eodory. The surface $(S, f)$ is hyperbolic by Lemma~\ref{L:h}.  	
For simplicity we assume that  the sequence $a_n$  monotonically approaches
$\infty$. 
There is a unique isometric (in the graph metric, where we identify multiple edges connecting any pair of vertices) embedding of the graph
in Figure~\ref{fig:EEE} into the graph in Figure~\ref{fig:EE} satisfying the following properties. The embedding extends to an orientation preserving homeomorphism of the plane, it takes the vertical linear subgraph that is the boundary of the face labeled 1 to the subgraph with the same properties, and it takes the top most horizontal face labeled $b$ to the face with the same properties.
Let $(S_n, f_n, w_n)$ be the
pointed sequence corresponding to Figure~\ref{fig:EE} with $w_n$ being the point
that corresponds to $w$ under the above isometric embedding of graphs.
Note that for each compact $K$ in $S$, there is an isometric embedding of
$K$ into $S_n$ for all $n$ large enough. Indeed, if $K$ is a compact in $S$,  it follows immediately that its projection to $\oC$ under $f$ cannot contain $\infty$. Thus, for all $n$ large enough, $a_n$ will be in the same connected component of $\oC\setminus f(K)$, and the claim follows.


Now, let  $S_n'$ be the connected component of the surface
obtained from $S_n$ by cutting out all the preimages of the extended
real line (this is our base curve) between $a_n$ and $\infty$ in the
first quadrant of Figure~\ref{fig:EE}, i.e., the part bounded by the vertical linear subgraph that is the boundary of the face labeled 1 and above the top most face labeled $b$, and that contains $w_n$. Each $(S_n',
f_n, w_n)$ is still a simply connected surface spread over $\oC$, even $\C$.  (It is,
however, not a log-Riemann surface in the sense of~\cite{BPM15} because its
completion is not obtained by adding a discrete set of points to $S$.)
Note that each $S_n',\ n=1,2,\dots$, is a subset of $S_{n+1}'$  because we
assume monotonicity of $a_n$, and also a subset of $S$. 
These facts along with the Schwarz Lemma imply that for
each $n,\ R_n\leq R_{n+1}\leq R$. In particular, $\lim R_n$ exists.
From the definition of $S_n'$ and the above claim on embedding every compact $K$ in $S$ into $S_n,\ n>N$, for some $N\in\N$, it follows that each such $K$ embeds into $S_n'$ for all $n$ large enough. Indeed, the compact $f(K)$ does not contain the segment between $a_n$ and $\infty$ for all large $n$. In particular, $(S_n',f_n, w_n),\ n\in\N$, has the same Carath\'eodory kernel $(S,f,w)$ as $(S_n,f_n, w_n),\ n\in\N$, and $\lim
R_n=R$. 
The proof of uniform convergence of $F_n$ to $F$ on compacta now follows the same lines as the proof of the necessity part of Theorem~\ref{T:CC}; see also~\cite[Theorem~1.1]{BPM15}.  
\end{proof}

\section{Examples}\label{S:TC}

\noindent
In this section we provide examples of Carath\'eodory convergence of parabolic surfaces to a hyperbolic one and vice versa, when one singular value approaches another. The following theorem facilitates the understanding of what happens to Speiser graphs under such convergence. 
\begin{theorem}\label{T:RE}
Let $(S_n, f_n)$ be a sequences of surfaces spread over the sphere from Speiser class $\mathcal S$ such that they all have isometric labeled Speiser graphs with singular values written in a cyclic order with respect to the same base curve as  $a_1, a_2, \dots, a_{k-1}$, $a_k(n)$, and $a_k(n)\to a_1,\ n\to\infty$. Assume further that $(S_n, f_n, w_n),\ n\in\N$, converges in the sense of Carath\'eodory to a surface $(S, f, w)$ as $n\to\infty$ in such a way that $w_n$ belongs to the closure of a hemisphere that corresponds to the same vertex under the isometry of Speiser graphs and $f_n(w_n)=f(w)\neq a_1$. Then  the surface $(S, f)$ belongs to the Speiser class $\mathcal S$ and its singular values are $a_1, a_2, \dots, a_{k-1}$. Moreover, its Speiser graph is the connected component of the graph  obtained from the common Speiser graph of the surfaces $(S_n, f_n)$ by removing the edges between $a_k(n)$ and $a_1$, and the connected component contains a vertex whose corresponding hemisphere contains $w$ in its closure.
\end{theorem}
\begin{proof}
The proof follows immediately from the definition of Ca\-ra\-th\'e\-o\-do\-ry convergence since we can add the full preimage of $a_1$ to the exceptional set $E$. In this case, if $K$ is any compact subset of $S\setminus E$, there exists $N\in\N$, such that no $\phi_{K,n}(K),\ n>N$, from Definition~\ref{D:Ker} contains any preimage of the arc of the base curve between $a_k(n)$ and $a_1$.
\end{proof}


\begin{figure}[h]
  \includegraphics[height=8cm]{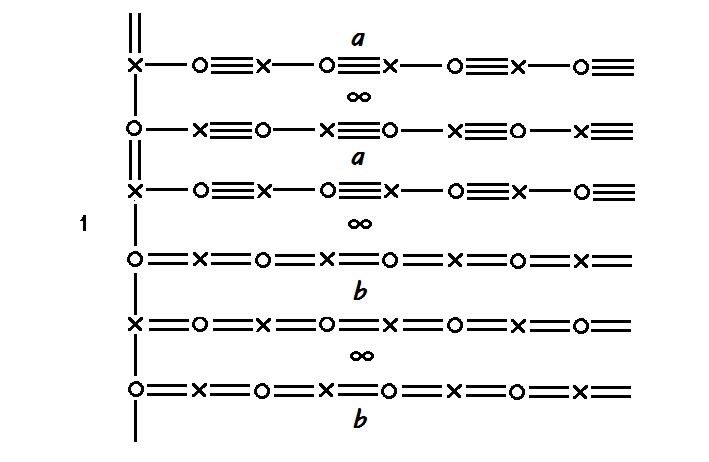}
  \caption{Double exponential, perturbed, $(S_{a,b}, f_{a,b})$.}
  \label{fig:EE}
 \end{figure}

\begin{figure}[h]
  \includegraphics[height=8cm]{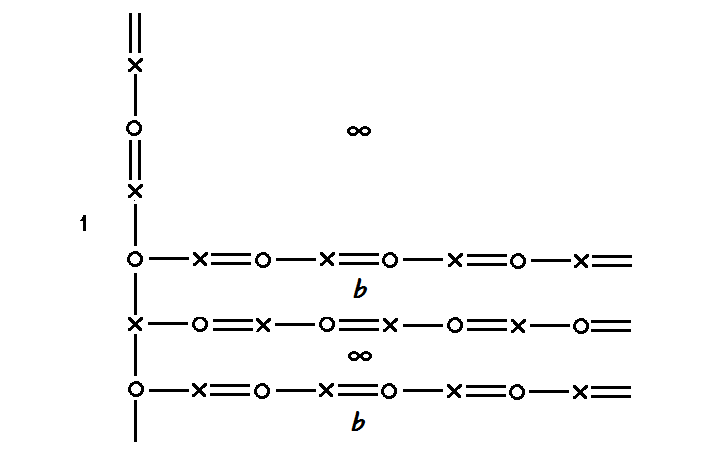}
  \caption{Hyperbolic surface, $(S_b, f_b)$.}
  \label{fig:EEE}
 \end{figure}
 
 \begin{figure}[h]
 	\includegraphics[height=8cm]{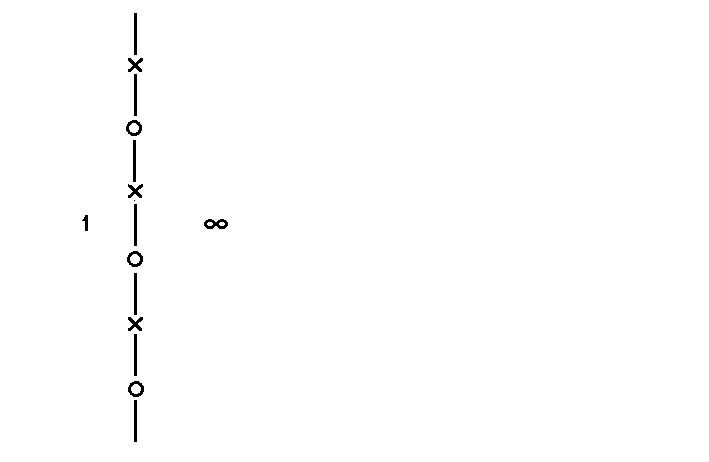}
 	\caption{Surface of the exponential function $f(z)=e^z+1$, $(S, f)$.}
 	\label{fig:E}
 \end{figure}
 
As an application, we obtain that the graph in Figure~\ref{fig:EEE}  is obtained from that in Figure~\ref{fig:EE} by sending $a=a_n$ to $\infty$, and the graph in Figure~\ref{fig:E} is obtained from the graph in Figure~\ref{fig:EEE} by sending $b=b_n$ to $\infty$.

 \begin{lemma}\label{L:h}
 The surface $(S_{a,b},f_{a,b})$ whose labeled Speiser graph is depicted in Figure~\ref{fig:EE} is parabolic for each $a, b\neq 1, \infty$.
 The surface $(S_b,f_b)$ with labeled Speiser graph from Figure~\ref{fig:EEE} is hyperbolic for each $b\neq 1,\infty$.
 The surface $(S,f)$ with labeled Speiser graph from Figure~\ref{fig:E} is parabolic.
 \end{lemma}
 \begin{proof}
 For $a=b=0$, the surface with labeled Speiser graph in Figure~\ref{fig:EE} is the surface of the double exponential function $z\mapsto\exp(\exp(z))$, and therefore is parabolic. For arbitrary $a, b\neq 0, \infty$, it is obtained from the double exponential  using a quasiconformal deformation, and so it is parabolic as well. An alternative way to conclude parabolicity for an arbitrary $a, b\neq 0, \infty$, 
is to use the result that for a surface spread over the sphere with finitely many singular values the conformal type only depends on the Speiser graph and not the labeling; see~\cite{Do84, Me03}.
 
 To show hyperbolicity of the surface whose labeled Speiser graph is depicted in Figure~\ref{fig:EEE}, with $b=0$, we use cutting and gluing techniques of Vol\-ko\-vys\-ki\u\i~\cite{Vo50}; see~\cite{GM05} for a similar example. 
 Namely, we make a horizontal cut that separates the graph into two parts and such that the cut crosses the lowest vertical edge that has the property that there are no asymptotic tracts above the cut that have asymptotic value $b=0$. The part above the cut is uniformized by the upper half-plane by the exponential map adjusted so that the asymptotic values are 1 and $\infty$ rather than 0 and $\infty$, namely by the map $f(z)=e^z+1$; see Figure~\ref{fig:E} for the labeled Speiser graph of $(\C, f(z))$. The part below the cut is uniformized by the double exponential map $z\mapsto \exp(\exp(z))$. Its labeled Speiser graph is given in Figure~\ref{fig:EE} with $a=0$.
The gluing homeomorphism of the real line is therefore given by $x\mapsto \ln\left(\ln\left(e^x+1\right)\right)$. This map is asymptotic to the identity as $x\to-\infty$ and to $x\mapsto\ln x$ as $x\to+\infty$. Therefore, according to~\cite[Theorm~24, p.~89]{Vo50}, the surface is hyperbolic. 

For arbitrary $b\neq 1,\infty$, the corresponding surface is obtained from a quasiconformal deformation of the above hyperbolic surface, and hence is also hyperbolic.

The surface with labeled Speiser graph in Figure~\ref{fig:E} is that of $f(z)=e^z+1$, and so is parabolic.
 \end{proof}

\begin{theorem}\label{T:PtoHtoP} For $a, b\in\R$, let $(S_{a,b}, f_{a,b})$ be a surface spread over the sphere whose labeled Speiser graph is depicted in Figure~\ref{fig:EE}. Fix an arbitrary point $w$ in an open hemisphere represented by either $\circ$ or $\times$ in Figure~\ref{fig:EEE}. Then, 
for any real sequence $a_n\to+\infty$, there exists a corresponding sequence of points $w_n\in S_{a_n,b},\ b\in\R$, such that the sequence of parabolic surfaces $(S_{a_n,b}, f_{a_n,b}, w_n),\ n\in\N$, whose labeled Speiser graphs are as in Figure~\ref{fig:EE} with $a=a_n$, converges in the sense of Carath\'eodory to the hyperbolic surface $(S_b,f_b, w)$ with labeled Speiser graph as in Figure~\ref{fig:EEE}. 

Likewise, for any point $w$ in an open hemisphere represented by either $\circ$ or $\times$ in Figure~\ref{fig:E}, and any real sequence $b_n\to+\infty$, there exists a corresponding sequence of points $w_n\in S_{b_n}$, such that 
the sequence of hyperbolic surfaces $(S_{b_n}, f_{b_n}, w_n),\ n\in\N$, with labeled Speiser graphs as in Figure~\ref{fig:EEE} with $b=b_n$ converges in the sense of Carath\'eodory to the parabolic surface $(S, f, w)$ with labeled Speiser graph depicted in Figure~\ref{fig:E}. 
\end{theorem}
\begin{proof}
There is an obvious isometric (in the graph metrics; multiple edges between two vertices being identified) embedding of the Speiser graph of $(S_b, f_b)$ into that of $(S_{a,b},f_{a,b})$ and the Speiser graph of  $(S, f)$ into the Speiser graph of $(S_b, f_b)$, for all $a\le b<1$. The embeddings are unique up to vertical shifts.  

For an arbitrary point $w\in S_b$ as in the first part of the statement, let $w_n\in S_{a_n,b}$ be the point that corresponds to $w$ under the isometric embedding of the Speiser graphs above.
Let $K$ be an arbitrary compact subset of $S_b$ that contains $w$. Let $\delta>0$ be small so that the closed disk $\overline D(\infty, \delta)$ in $\oC$ centered at $\infty$ of radius $\delta$ does not contain either $1$ or $b$.  We choose the extended real line to be a base curve $\beta$. 
 For singular values $1, b, a_n,\infty$ of $f_{a_n,b}$,  the graph $G_\beta$ defined in Subsection~\ref{SS:SSS} is embedded in $\oC$, has two vertices $\times$ and $\circ$ and four edges connecting them, each crossing the base curve $\beta$ between two adjacent singular values. 
We can further choose $\delta$ such that $\overline D(\infty, \delta)$ does not intersect $G_\beta$.
The full preimage $U_\delta=f_b^{-1}(D(\infty, \delta))$ then consists of infinitely many connected components $U_{k}^{\infty}(\delta),\ k\in\N$, that are open topological half-planes. 
Finally, we may choose $\delta>0$ even smaller so that none of the components $U_{k}^{\infty}(\delta),\ k\in\N$, intersects the compact $K$. 
Indeed, the family of open sets 
$$
V_{b, \delta}=f_b^{-1}\left(\oC\setminus\overline D(\infty, \delta)\right),\quad \delta>0, 
$$
forms an open cover of $S_b$ and hence of $K$.
If small $\delta>0$ is chosen such that the above conditions are satisfied, choosing $N$ such that $a_n\in B(\infty,\delta),\ n\ge N$, works. To see this, we just observe that for $n\ge N$, the isometric embedding of Speiser graphs above induces an embedding $i_n$ of  $V_{b,\delta}$ into $S_{a_n,b}$ such that  $f_{a_n,b}\circ i_n=f_b$. 

Since the above holds for any sequence $(S_{a_n,b}, f_{a_n,b}, w_n),\ a_n\to+\infty$, and the maximality of $(S_b, f_b, w)$ is trivial, the first part of the theorem follows.

The convergence of $(S_{b_n}, f_{b_n}, w_n)$ to $(S, f, w)$ follows the same lines.
\end{proof}

\section{Changing the order}\label{S:CO}

\noindent
In this section we show, in addition, that convergence in the sense of Carath\'eodory can change the order of entire functions.
\begin{theorem}\label{T:TC}
There exists a sequence of normalized triples $(\C, f_n, 0)$, $n\in\N$, where each $f_n$ is an entire function of infinite order, that converges in the sense of Carath\'eodory to a normalized triple $(\C, f, 0)$, where $f$ has order 1.
\end{theorem}
\begin{proof}
Consider labeled Speiser graphs in Figure~\ref{fig:EE} with $a=b=a_n\to\infty$, and denote the corresponding surfaces by $(S_n, f_n)$. Note that in this case there are only 3 singular values, namely $1, a_n,\infty$, and so some of the double edges in Figure~\ref{fig:EE} are identified to become single edges. Also, let $f(z)=e^z+1$. This is an entire function of order 1 whose Speiser graph is depicted in Figure~\ref{fig:E}.
We choose a ``vertical'' isometric embedding of the Speiser graph in Figure~\ref{fig:E} to that in Figure~\ref{fig:EE}, i.e., an embedding such that one of the complementary components of the embedded graph does not contain any vertices of the graph in Figure~\ref{fig:EE}. Such an embedding is unique up to a vertical translation. The point 0 in $\C$ is on the common boundary of two of the components of the preimages of the upper and lower hemispheres, in this case half-planes, under $f$. Such half-planes correspond to the vertices of the Speiser graph in Figure~\ref{fig:E}. Let $w_n$ be the point in $S_n$ that corresponds to 0 under the embedding of half-planes that corresponds to the embedding of Speiser graphs.   

The surface $(S_n, f_n, w_n)$ is 
equivalent to $(\C, f_n, 0)$, where $$f_n(z)=a_n\left(\exp(\exp(z))-1\right)+1,$$ and so $f_n$ has infinite order. 
Here, $w_n=\ln\ln\left(1+1/a_n\right)$.
Arguing as in Theorem~\ref{T:PtoHtoP}, we conclude that $(\C, f_n,0),\ n\in\N$, converges in the sense of Carath\'eodory to the surface $(\C, f, 0)$ as $a_n\to\infty$.
In this case the  maps $\phi_{K,n}=f_n^{-1}\circ f$ are asymptotic to translations $z\mapsto z+w_n$. The theorem follows.

Note that, due to~\cite{Bi15}, we could not argue that the orders of functions in the sequence are infinite based on the fact that the corresponding surfaces could be obtained  by quasiconformal deformations from the surface of the double exponential function.
\end{proof}


\end{document}